\documentclass[11pt,letterpaper,oneside]{amsart}

\usepackage{amsmath,amsthm,amsfonts,amssymb}
\usepackage{mathtools}
\usepackage{mathrsfs}
\usepackage{slashbox}
\usepackage[active]{srcltx}
\usepackage[pdftex,bookmarks=true]{hyperref}
\usepackage{fullpage}
\usepackage{wasysym}
\usepackage{enumerate}
\usepackage{multirow}
\usepackage{graphicx}
\usepackage{multicol}

\newcommand{\spc}{\text{ }}
\newcommand{\calc}{\mathcal{C}}
\newcommand{\tab}{\hspace{0.2in}}

\newcommand{\co}{\mathscr{O}}
\newcommand{\costar}{\mathscr{O}^*}

\numberwithin{equation}{section}
\theoremstyle{plain}
\newtheorem{thm}[equation]{Theorem}
\newtheorem{cor}[equation]{Corollary}
\newtheorem{lem}[equation]{Lemma}

\newtheorem{conj}[equation]{Conjecture}

\theoremstyle{definition}
\newtheorem{defn}[equation]{Definition}

\theoremstyle{remark}

\begin{document}
\belowdisplayskip=12pt
\belowdisplayshortskip=7pt plus 3pt
\title{\large Weak Orientability of Matroids \\ and Polynomial Equations}

\author{J.A. De Loera, J. Lee, S. Margulies, J. Miller}
\email{}

\date{\today}

\begin{abstract} This paper studies systems of polynomial equations that provide information about orientability of matroids.

First, we study systems of linear equations over $F_2$, originally alluded to by Bland and Jensen in their seminal paper on weak orientability. The Bland-Jensen linear equations for a matroid $M$ have a solution if and only if $M$ is weakly orientable. We use the Bland-Jensen system to determine weak orientability for all matroids on at most nine elements and all matroids between ten and twelve elements having rank three.  Our experiments indicate that for small rank, about half the time, when a simple matroid is not orientable, it is already non-weakly orientable. Thus, about half of the small simple non-orientable matroids of rank three are not representable over fields having order congruent to three modulo four.
For  binary matroids, the Bland-Jensen linear systems provide a practical way to check orientability.

Second, we present two extensions of the Bland-Jensen equations to slightly larger systems of \emph{non-linear}  polynomial equations.
Our systems of polynomial equations have a solution if and only if the associated matroid $M$ is orientable. The systems come in two versions, one directly
extending the Bland-Jensen system for $F_2$, and a different system working over other fields. We study some basic algebraic properties of these systems.

Finally, we present an infinite family of non-weakly-orientable matroids, with growing rank and co-rank. We conjecture that these matroids are minor-minimal
non-weakly-orientable matroids.
\end{abstract}

\maketitle

\section{Introduction}
Oriented matroids are special matroids with additional geometric-topological structure.  Their importance in mathematics is demonstrated by their generalizing of many combinatorial objects including directed graphs, vector configurations over ordered fields, hyperplane arrangements in Euclidean space, convex polytopes, and even linear programs (see \cite{basic} and \cite{Orientmat}). Oriented matroids  were first introduced by Bland and Las Vergnas in \cite{Matorient} and Folkman and Lawrence in \cite{Matorient2}. Not all matroids can be oriented, and deciding when a matroid is orientable is related to the classical problem of realizability of matroids over the real numbers that has found many applications in computational geometry (see e.g., \cite{Orientmat,Bokowskietal,bokowskisturmfelsbook} and references therein). Determining whether a matroid, given by its set of non-bases, is orientable is  an NP-complete problem, even for fixed rank (see \cite{OrientNP}). In this paper, we revisit the problem of deciding orientability, through the simpler notion of weak orientability.
As we will recall, weakly-oriented matroids are a natural intermediate structure between matroids and oriented matroids.

Weakly-oriented matroids were introduced by Bland and Jensen in the late 1980's in \cite{Weak}.
Their starting point was a theorem of Minty characterizing matroids via a coloring property
on the sets of circuits and cocircuits (see \cite{Minty}). Oriented matroids satisfy a stronger version of this property
that is defined on the sets of \emph{signed} circuits and cocircuits.
There is a natural intermediate property, also defined on the sets of
signed circuits and cocircuits, which defines what Bland and Jensen called weakly-oriented matroids.
Bland and Jensen alluded to the existence of a system of linear equations over $F_2$ that could be
used to determine the weak orientability of a given matroid \cite{Weak}.

Note that while non-weak orientability is easier to verify (as it only needs the solution of a linear
system of equations), it is also a more pathological condition than non-orientability for matroids. Indeed,
non-weak orientability implies non-representability over fields having order congruent to 3 modulo 4 (see \cite{Weak});
in particular, ternary matroids (i.e., those representable over $F_3$) are always weakly-orientable, while ternary matroids
may or may not be orientable (see \cite{LeeScobee}). Moreover if a matroid is ternary and orientable,
then it is representable over every field that does not have characteristic two (see \cite{LeeScobee}).
Unfortunately, while much work has been done in the field of oriented matroids since their introduction,
the study of  weakly-oriented matroids has remained relatively untouched. In this paper, we continue 
the path outlined by Bland and Jensen with three main contributions:

\begin{itemize}

\item After we give an explicit way to write what we suppose are the linear equations  alluded to
 by Bland and Jensen in \cite{Weak}  from the circuits and cocircuits of the input matroid, we
extend the original linear system to a slightly larger  system with new higher-degree polynomial equations. Our new system can be used to detect orientability.
A matroid is orientable if and only if our system has a solution  over $F_2$. We also present another generalization for fields of characteristic different from two.

\item The orientability of matroids with no more than nine elements and no more than twelve elements with rank-three case was  investigated in \cite{10ele}.
The techniques of \cite{10ele} use a transformation of the orthogonality axioms of oriented matroids to a SAT problem. In the present paper,
we use instead the Bland-Jensen linear equations to determine weak orientability for all of these matroids.  The classification resulting from our computations indicate
non-orientable matroids are often non-weakly-orientable already. Our experiments indicate that for small rank, about half the time, when a simple matroid is not orientable, it is already non-weakly orientable.

%

\item Finally, it is well known that orientability cannot be described by a finite list of excluded minors: an infinite minor-minimal family of growing rank and co-rank was originally presented by Bland and Las Vergnas in \cite{Matorient}, and a infinite minor-minimal family of fixed-rank three was later presented by Ziegler in \cite{MinZieg}. It was conjectured by Bland and Jensen in their paper that the same is true for non-weak orientability \cite{Weak}. We present an infinite family of  matroids, with increasing rank and co-rank, that are all non-weakly orientable, and conjecture that they are in fact minor-minimal.
\end{itemize}


\section{Weak Orientability --- An overview} \label{introtowoms}

We assume elementary knowledge of matroid and oriented matroid theory; for further reading, see  \cite{Orientmat} and \cite{OxleyBook}. For a \emph{matroid} $M$, let $E(M)$ denote the ground set of $M$, and let $\calc(M)$ denote the set of \emph{circuits} of  $M$. Also, let $\calc^*(M)$ denote the set of \emph{cocircuits} of $M$, or equivalently the set of circuits of $M^*$, the \emph{dual} matroid of $M$. Similarly, for an \emph{oriented matroid} $M$, we let $E(M)$ denote the ground set of $M$, and we let
$\co(M)$  and $\costar(M)$ denote, respectively, the signed circuits and cocircuits of $M$.

Let $E$ be a finite set. A \emph{signed subset} $Z$ of $E$ is a pair $Z=(Z^+ ,Z^- )$ such that
$Z^+$ and $Z^-$ partition a subset of $E$ (i.e., $Z^+\subseteq E$, $Z^-\subseteq E$,
and $Z^+\cap Z^-=\varnothing$).
Here we define the signed subset $-Z:=(Z^-,Z^+)$ and the set $\underbar{Z}:=Z^+\cup Z^-\subseteq E$. For a collection $S$ of signed subsets of $E$, we denote
by $-S=\{-X  ~\big|~  X \in S\}$.

Let $\co$ and $\costar$ each be collections of signed subsets of $E$ such that:
 \begin{itemize}
\item $-\co=\co$,
\item $-\costar=\costar$,
\item if $X,Z\in \co$ with $\underbar{Z}\subseteq\underbar{X}$, then $Z=\pm X$,
\item if $Y,Z\in \costar$ with $\underbar{Z}\subseteq\underbar{Y}$, then $Z=\pm Y$,
 \end{itemize}
 For such collections, we define $\underline{\co}:=\{\underbar{X} ~\big|~ X\in\co\}$  and
 $\underline{\costar}:=\{\underbar{Y} ~\big|~ Y\in\costar\}$ . For a matroid $M$ with $\underline{\co}=\calc(M)$ and $\underline{\costar}=\calc^*(M)$ we say that $\co$ and $\costar$ are a \emph{signing of circuits and cocircuits} of $M$.

Following \cite{Matorient}, we say that a pair of signed subsets $X$ and $Y$ of $E$ are \emph{orthogonal} if
\[
(X^+\cap Y^+)\cup (X^-\cap Y^-)\neq\varnothing\iff (X^+\cap Y^-)\cup (X^-\cap Y^+)\neq\varnothing.
\]
That is, $X$ and $Y$ agree in sign on some element if and only if they disagree
in sign on some element. Moreover, it is well known (see \cite{Matorient}) that when $\underline{\co}$ and $\underline{\costar}$
are the sets of circuits and cocircuits, respectively, of a matroid, then
enforcing the orthogonality condition on all $X$ and $Y$ having $|\underline{X}\cap \underline{Y}|\in\{2,3\}$
is equivalent to enforcing the orthogonality condition on all $X$ and $Y$
(with no condition on $|\underline{X}\cap \underline{Y}|$).

In what follows, we will use the following characterization of weak orientability, which, owing to the remarks
above, is a weakening of the orthogonality condition characterizing oriented matroids. In \cite{Weak}, the following 
definitions were motivated from the Minty Coloring Property for matroids (see \cite{Minty}) and derived as a theorem:

\begin{defn}
Let $\co$ and $\costar$ be a set of signings of the circuits and cocircuits, respectively, of a matroid $M$. We say
the triple $(E,\co,\co^*)$ determines a dual pair of \emph{weakly-oriented matroids}
if for every $X\in\co$ and $Y\in\co^*$ with $|\underbar{X}\cap\underbar{Y}|=2$, we have that $X$ and $Y$ are orthogonal.

In this case $\co$ and $\costar$ are the sets of signed circuits and cocircuits, respectively, of a weakly-oriented matroid (with underlying matroid $M$).

\end{defn}

\begin{defn}\label{weakdef2}
If $M$ is a matroid such that for some triple $(E,\co,\co^*)$ that determines a dual pair of weakly-oriented matroids we have $E(M)=E$, $\calc(M)=\underline{\co}$, and $\calc^*(M)=\underline{\costar}$, then
we say that $M$ is a \emph{weakly-orientable matroid}.
\end{defn}

Thinking of  orientability as a sufficient condition for weak orientability, Bland and Jensen gave
us a stronger (and sometimes more easily checked) sufficient condition for  weak orientability:

\begin{cor}[Bland and Jensen \cite{Weak}] If a matroid $M$ is representable over a
field having order  congruent to 3 modulo 4, then $M$ is weakly-orientable.
\end{cor}

In particular, all matroids representable over $F_3$ are weakly-orientable (but, we note, not always orientable). More generally matroids representable over a 
finite field in which $-1$ is not a square are weakly orientable \cite{wagowski}. We also have a description of weak orientability purely in terms of circuits and cocircuits 
which will be very important for our story:

\begin{thm}[Bland and Jensen \cite{Weak}]\label{eq:weakeq} A matroid $M$ is weakly-orientable if and only if there exists no list of odd cardinality $\{(X_i,Y_i)\}_{i=1}^n$ such that
\begin{enumerate}[i)]
\item $X_i\in\calc(M)$ and $Y_i\in\calc^*(M)$, $\forall i$,
\item $|X_i\cap Y_i|=2$, $\forall i$,
\item For all pairs $X\in\calc(M)$ and $e\in X$ we have $|\{i ~\big|~ X_i=X, e\in X_i\cap Y_i\}|$ is even,
\item For all pairs $Y\in\calc^*(M)$ and $e\in Y$ we have $|\{i ~\big|~ Y_i=Y, e\in X_i\cap Y_i\}|$ is even.
\end{enumerate}
\end{thm}

The outline of the proof of the above theorem given by Bland and Jensen alluded to a linear system over $F_2$ associated to $M$ that is feasible if and only if $M$ is weakly orientable (see \cite[page 12]{Weak}). We proceed here to give an explicit description of such a system, which we suppose are the equations that Bland and Jensen alluded to. Afterwards, we continue on with some consequences.

We begin by reformulating our definitions as in \cite{Lee} (also see \cite[Proposition 6.1]{repthm}):
\begin{defn} Let $\Sigma=\{0,+,-\}$. A \emph{$\Sigma$-mapping} for a matroid $M$ is a pair of maps
$$\phi_{\Sigma}:E(M)\times \calc(M)\mapsto \Sigma,$$
$$\phi_{\Sigma}^*:E(M)\times \calc^*(M)\mapsto \Sigma,$$
such that
$$\phi_{\Sigma}(e,X)=0\spc\text{iff}\spc e\notin X,\spc\forall e\in E(M),\spc X\in\calc (M),$$
$$\phi_{\Sigma}^*(e,Y)=0\spc\text{iff}\spc e\notin Y,\spc\forall e\in E(M),\spc Y\in\calc^* (M).$$
\end{defn}

\begin{lem}\label{eq:weakorlem} A matroid $M$ is \emph{weakly-orientable} if and only if $M$ has a $\Sigma$-mapping $(\phi_{\Sigma},\phi_{\Sigma}^*)$ such that for every pair $X\in \calc(M)$ and $Y\in \calc^*(M)$ with $|X\cap Y|=2$,  the following holds:
$$\phi_{\Sigma}(e,X)\cdot\phi_{\Sigma}^*(e,Y)>0,\text{ for some }e\in X\cap Y$$
$$\iff\spc\spc\phi_{\Sigma}(f,X)\cdot\phi_{\Sigma}^*(f,Y)<0,\text{ for some }f\in X\cap Y.$$
\end{lem}

\begin{proof} Given such a $\Sigma$-mapping, let $\co=\big\{\pm\big(\phi_{\Sigma}(\cdot,X)^{-1}(+),\phi_{\Sigma}(\cdot,X)^{-1}(-)\big)\big| X\in\calc(M)\big\}$ and $\co^*=\big\{\pm\big(\phi_{\Sigma}^*(\cdot,Y)^{-1}(+),\phi_{\Sigma}^*(\cdot,Y)^{-1}(-)\big)\big| Y\in\calc(M)\big\}$. It can be easily checked that $(E(M),\co,\co^*)$ determines a dual pair of weakly orientable matroids with $\underline{\co}=\calc(M)$ and $\underline{\co^*}=\calc^*(M)$. Conversely, given a dual pair of weakly orientable matroids with $\underline{\co}=\calc(M)$ and $\underline{\co^*}=\calc^*(M)$ and $E=E(M)$, for every $X\in\calc(M)$ choose an $X'\in\co$ with $\underline{X'}=X$ and similarly for every $Y\in\calc^*(M)$ choose a $Y'\in\co$ with $\underline{Y'}=Y$. Define for every $X\in\calc(M)$ and $e\in E$,
\[
\phi_{\Sigma}(e,X)=\begin{cases} 0 & e\notin X,\\
+ & e\in (X')^+,\\
- & e\in (X')^-.
\end{cases}
\]
Defining $\phi_{\Sigma}^*$ similarly, we have a  $\Sigma$-mapping that satisfies the conditions of the lemma.
\end{proof}

We now define what we call the  \emph{Bland-Jensen (linear) system} (named in recognition of the seminal work of R.G. Bland and D.L. Jensen).

\begin{thm}\label{eq:polyweak} Let $M$ be a matroid. We define variables $a_{e,X}$ for each pair $X\in \calc(M)$ and $e\in X$ and  variables $b_{e,Y}$
for each pair $Y\in \calc^*(M)$ and $e\in Y$. Then consider the linear system over $F_2$ consisting of the equations
\[
a_{e,X}+b_{e,Y}+a_{f,X}+b_{f, Y}+1=0, \tag{$g_{X,Y}$}
\]
for every $X\in \calc(M)$ and $Y\in \calc^*(M)$ such that $X\cap Y=\{e,f\}$. Then $M$ is weakly-orientable if and only if this linear system is feasible over $F_2$.
\end{thm}
\begin{proof} Let $\pi: (\{+,-\},\cdot)\rightarrow (F_2,+)$ be the natural group isomorphism from the multiplicative group on $\{+,-\}$ to the underlying addive group of $F_2$. For a given $\Sigma$-mapping  $(\phi_{\Sigma},\phi_{\Sigma}^*)$ it satisfies the conditions of Lemma \ref{eq:weakorlem} if and only if for every $X\in \calc(M)$ and $Y\in \calc^*(M)$ such that $X\cap Y=\{e,f\}$ we have
$$\phi_{\Sigma}(e,X)\phi^*_{\Sigma}(e,Y)\phi_{\Sigma}(f,X)\phi_{\Sigma}^*(f,Y)=-\spc\iff$$
$$\pi\big(\phi_{\Sigma}(e,X)\phi^*_{\Sigma}(e,Y)\phi_{\Sigma}(f,X)\phi_{\Sigma}^*(f,Y)\big)=\pi(-)=1\spc \iff$$
$$\pi(\phi_{\Sigma}(e,X))+\pi(\phi^*_{\Sigma}(e,Y))+\pi(\phi_{\Sigma}(f,X))+\pi(\phi_{\Sigma}^*(f,Y\big))+1=0.$$
Therefore making the association $a_{e,X}=\pi(\phi_{\Sigma}(e,X))$, $b_{e,Y}=\pi(\phi^*_{\Sigma}(e,Y))$ for all our variables (and in the direction
of going from variables to a $\Sigma$-mapping, filling in the remaining values of our $\Sigma$-mapping with zeros), the result is clear.
\end{proof}

We now use this to produce a full proof of Theorem \ref{eq:weakeq} and show that the above linear system is functionally equivalent to the one alluded to
by Bland and Jensen. We require the following lemma which is a simple case of the Fredholm Alternative Theorem (see, for example, \cite{Schrijver1986}) and leave the details for
the reader:

\begin{lem}\label{eq:bifark} Let $A$ be an $m\times n$ matrix over $F_2$ and let $b\in F_2^m$. There exists no solution to the system $Ax=b$ if and only if there exists $y\in F_2^m$ such that $y^T A={\bf 0}$ and $y^T b=1$.
\end{lem}

\begin{proof}[Proof of  Theorem \ref{eq:weakeq}:] A matroid $M$ is non-weakly-orientable if and only
if the linear system over $F_2$ in Theorem \ref{eq:polyweak} (henceforth denoted $Ax=b$) is infeasible.
 This is the case, according to Lemma \ref{eq:bifark},
 if and only if there exists a vector $y$ such that $y^TA={\bf 0}$ and $y^Tb=1$.

Assume that such a $y$ exists; each component where $y$ is 1 identifies an equation of
our linear system, which in turn corresponds to a circuit/cocircuit pair $(X_i,Y_i)$ with $|X_i\cap Y_i|=2$.
Because $y^Tb=1$, we have that there is an odd number of equations so identified.
Further, because $y^TA={\bf 0}$, we have that $|\{i ~\big|~ y_i=1\text{ and } A_{ij}=1\}|$ is even for all $j$.
This means that for any pair $X_i$ and $e\in X_i$ which corresponds to some variable $a_{e,X_i}$ and
hence some $j$, that $|\{i ~\big|~ X_i=X, e\in X_i\cap Y_i\}|$ is even. Similarly for any pair $Y_i$ and
$e\in Y_i$ we have $|\{i ~\big|~ Y_i=Y, e\in X_i\cap Y_i\}|$ is even. Thus a list as described in
Theorem \ref{eq:weakeq} exists.

Conversely, assume such a list exists. Then, by letting $y$ be the vector that is one for all equations
corresponding to pairs $(X_i,Y_i)$ in this list and zero elsewhere, it is  straightforward to check that the
specifics of this list imply that $y^TA={\bf 0}$ and $y^Tb=1$.

Therefore $M$ is non-weakly-orientable if and only if a list of the type described in Theorem \ref{eq:weakeq} exists, which completes the proof.
\end{proof}

Later, in Section \ref{tablesofwoms}, we will use the Bland-Jensen linear system to classify which matroids of small sizes are weakly-orientable.

\section{Extending the Bland-Jensen system for orientability testing}

We begin by recalling the definition of orientability in terms of $\Sigma$-mappings from the previous section, which is easily seen to be equivalent to other definitions of orientability (see \cite{Lee}).

\begin{defn}\label{eq:ordef} A matroid $M$ is \emph{orientable} if $M$ has a $\Sigma$-mapping $(\phi_{\Sigma},\phi_{\Sigma}^*)$ such that for every pair $X\in \calc(M)$ and $Y\in \calc^*(M)$ where $|X\cap Y|\in\{2,3\}$, then the following holds:
$$\phi_{\Sigma}(e,X)\cdot\phi_{\Sigma}^*(e,Y)>0,\text{ for some }e\in X\cap Y$$
$$\iff\spc\spc\phi_{\Sigma}(f,X)\cdot\phi_{\Sigma}^*(f,Y)<0,\text{ for some }f\in X\cap Y.$$
\end{defn}
Similar to weak orientability, the relationship between signed circuits and cocircuits above corresponds to the usual orthogonality axiom in oriented matroid theory. Notably, for orientability, although it suffices to look at circuits and cocircuits with intersections of just two or three elements, the orthogonality property holds for all  pairs of circuits and cocircuits (see \cite[p.~118]{Orientmat}).

We aim to take the linear system associated to a matroid $M$ in Theorem \ref{eq:polyweak}, and extend it to a system of polynomial equations the solvability of which is equivalent to orientability of $M$. It is well-known that polynomials play an important role in the theory of oriented matroids (see \cite{bokowskisturmfelsbook,Orientmat}). E.g., we know the Grassmann-Pl\"ucker equations, describing  Grassmann variety, allows us to give description of \emph{realizable} matroids or oriented matroids. 
But the first in-depth study of polynomial systems describing  all oriented matroids was done by Bokowski, Richter-Gebert and Guedes de Oliveira in 
\cite{BORGGO}. They described an algebraic variety over $F_2$ whose points correspond to all matroids of given rank on a given finite set, and similarly for $F_3$ and oriented matroids. Although their equations have a different set of variables (they use bases, instead of circuits and cocircuits to label variables). An interesting open question is whether
one can derive a natural equivalence between the equations in \cite{BORGGO} and our equations presented next:

\begin{thm}\label{eq:poly2} Define the variables $a_{e,X}$ for each pair $X\in \calc(M)$ and $e\in X$, and define the variables $b_{e,Y}$ for each pair $Y\in \calc^*(M)$ and $e\in Y$. Then let $R$ be the polynomial ring on these variables over the field $F_2$. We consider the following polynomial equations in $R$:

\smallskip
For every $X\in \calc(M)$ and $Y\in \calc^*(M)$ such that $X\cap Y=\{e,f\}$~,
\[
\tag{$g_{X,Y}$} a_{e,X}+b_{e,Y}+a_{f,X}+b_{f, Y}+1=0,
\]
and for every $X\in \calc(M)$ and $Y\in \calc^*(M)$ such that $X\cap Y=\{e,f,g\}$~,
\[
1+a_{e,X}+b_{e,Y}+a_{f,X}+b_{f, Y}+a_{g,X}+b_{g, Y}+a_{e,X}a_{f,X}+a_{e,X}a_{g,X}+a_{e,X}b_{f,Y}+a_{e,X}b_{g,Y}+a_{f,X}a_{g,X}
\]
\[
\tab +a_{f,X}b_{e,Y}+a_{f,X}b_{g,Y}+a_{g,X}b_{e,Y}+a_{g,X}b_{f,Y}+b_{e,Y}b_{f,Y}+b_{e,Y}b_{g,Y}+b_{f,Y}b_{g,Y}=0.\tag{$h_{X,Y}$}
\]
\smallskip

Then $M$ is orientable if and only if this system of polynomial equations over $F_2$ is feasible.
\end{thm}

\begin{proof} First note for any $X\in \calc(M)$ and $Y\in \calc^*(M)$ with $X\cap Y=\{e,f,g\}$, we have over $F_2$ that
$$0=1+a_{e,X}+b_{e,Y}+a_{f,X}+b_{f, Y}+a_{g,X}+b_{g, Y}+a_{e,X}a_{f,X}+a_{e,X}a_{g,Y}+a_{e,X}b_{f,Y}+a_{e,X}b_{g,Y}+a_{f,X}a_{g,X}$$
$$\tab +a_{f,X}b_{e,Y}+a_{f,X}b_{g,Y}+a_{g,X}b_{e,Y}+a_{g,X}b_{f,Y}+b_{e,Y}b_{f,Y}+b_{e,Y}b_{g,Y}+b_{f,Y}b_{g,Y}$$
$$=(a_{e,X}+b_{e,Y}+a_{f,X}+b_{f,Y}+1)(a_{e,X}+b_{e,Y}+a_{g,X}+b_{g,Y}+1)(a_{g,X}+b_{g,Y}+a_{f,X}+b_{f,Y}+1)$$
if and only if one of the following holds:
\begin{align*}
a_{e,X}+b_{e,Y}+a_{f,X}+b_{f,Y}+1&=0,\\
a_{e,X}+b_{e,Y}+a_{g,X}+b_{g,Y}+1&=0,\\
a_{g,X}+b_{g,Y}+a_{f,X}+b_{f,Y}+1&=0.
\end{align*}
As in the proof of Theorem \ref{eq:polyweak}, let $\pi:(\{+,-\},\cdot)\rightarrow (F_2,+)$ be the natural isomorphism. We associate our variables to a $\Sigma$-mapping $(\phi_{\Sigma},\phi_{\Sigma}^*)$ by making the assignment $a_{e,X}=\pi(\phi_{\Sigma}(e,X))$, $b_{e,Y}=\pi(\phi_{\Sigma}(e,Y))$, and assigning all remaining values of the $\Sigma$-mapping to $0$ when going from variables to a $\Sigma$-mapping.

Under this association, by the proof of Theorem \ref{eq:polyweak}, we have the conditions of Definition \ref{eq:ordef} for a $\Sigma$-mapping and a pair $X\in\calc(M)$ and $Y\in\calc^*(M)$ with $|X\cap Y|=2$ are satisfied if and only if $g_{X,Y}=0$. Therefore it suffices to show that under this association the conditions of Definition \ref{eq:ordef} for a $\Sigma$-mapping  and a pair $X\in\calc(M)$ and $Y\in\calc^*(M)$ with $|X\cap Y|=3$ are satisfied if and only if $h_{X,Y}=0$. But by the above note, for a pair $X\in\calc(M)$ and $Y\in\calc^*(M)$ with $|X\cap Y|=3$, $h_{X,Y}=0$ if and only if there exists some pair $\{\alpha,\beta\}\subseteq X\cap Y$ such that
$$a_{\alpha,X}+b_{\alpha,Y}+a_{\beta,X}+b_{\beta,Y}+1=0,$$
which under our association, as shown in the proof of Theorem \ref{eq:polyweak}, is true if and only if
$$\phi_{\Sigma}(\alpha,X)\phi_{\Sigma}^*(\alpha,Y)\phi_{\Sigma}(\beta,X)\phi_{\Sigma}^*(\beta,Y)=-,$$
for some $\{\alpha,\beta\}\subseteq X\cap Y$. These are exactly the conditions of Definition \ref{eq:ordef} for a $\Sigma$-mapping  and a pair $X\in\calc(M)$ and $Y\in\calc^*(M)$ with $|X\cap Y|=3$. Thus we complete the proof.
\end{proof}

\begin{thm}\label{eq:poly} Define the variables $a_{e,X}$ for each pair $X\in \calc(M)$ and $e\in X$ and define the variables $b_{e,Y}$ for each pair $Y\in \calc^*(M)$ and $e\in Y$. Then let $R$ be the polynomial ring on these variables over a field $K$ (where $K$ is a field of characteristic not equal to $2$). We consider the following polynomials in $R$: for every $e\in X$ and $X\in \calc(M)$
\[
\tag{$p_{e,X}$} a_{e,X}^2-1=0,
\]
for every $e\in Y$ and $Y\in \calc^*(M)$
\[
\tag{$q_{e,Y}$} b_{e,Y}^2-1=0,
\]
for every $X\in \calc(M)$ and $Y\in \calc^*(M)$ such that $X\cap Y=\{e,f\}$
\[
\tag{$h'_{X,Y}$} a_{e,X}b_{e,Y}+a_{f,X}b_{f,Y}=0,
\]
and for every $X\in \calc(M)$ and $Y\in \calc^*(M)$ such that $X\cap Y=\{e,f,g\}$
\[
\tag{$h_{X,Y}$} a_{e,X}b_{e,Y}a_{f,X}b_{f,Y}+a_{f,X}b_{f,Y}a_{g,X}b_{g,Y}+a_{g,X}b_{g,Y}a_{e,X}b_{e,Y}+1=0.
\]
Then $M$ is orientable if and only if there exists a solution to this system of polynomials over the field $K$.
\end{thm}

\begin{proof} \fussy Assume $M$ is orientable, and let $(\phi_{\Sigma},\phi_{\Sigma}^*$) be its given $\Sigma$-mapping. For every $e\in X$ and $X\in \calc(M)$, we let $a_{e,X}=1$ if $\phi_{\Sigma}(e,X)=+$ and $a_{e,X}=-1$ if $\phi_{\Sigma}(e,X)=-$, we note that because $e\in X$ this implies $\phi_{\Sigma}(e,X)\neq 0$.

\fussy Since $1^2=1=(-1)^2$ in every field $K$, it is clear that for these choices of $a$'s and $b$'s $p_{e,X}=0$ for all $e\in X$ and $X\in \calc(M)$ and $q_{e,Y}=0$ for all $e\in Y$ and $Y\in \calc^*(M)$.

\fussy Then if $X\in \calc(M)$ and $Y\in \calc^*(M)$ with $X\cap Y=\{e,f\}$, because $\phi_{\Sigma}(e,X),\phi_{\Sigma}^*(e, Y)\neq 0$ (as $e\in X,Y$) we have that $\phi_{\Sigma}(e,X)\phi_{\Sigma}^*(e,Y)>0$ or $\phi_{\Sigma}(e,X)\phi_{\Sigma}^*(e,Y)<0$ and therefore that $\phi_{\Sigma}(f,X)\phi_{\Sigma}^*(f,Y)=-\phi_{\Sigma}(e,X)\phi_{\Sigma}^*(e,Y)$ because $M$ is orientable. Therefore  $a_{e,X}b_{e,Y}=1$ or $a_{e,X}b_{e,Y}=-1$ and $a_{f,X}b_{f,Y}=-a_{e,X}b_{e,Y}$, which implies that
$$h'_{X,Y}=a_{f,X}b_{f,Y}+a_{e,X}b_{e,Y}=0.$$

\fussy Then if $X\in \calc(M)$ and $Y\in \calc^*(M)$ with $X\cap Y=\{e,f,g\}$, we have that $a_{e,X}b_{e,Y}$, $a_{f,X}b_{f,Y}$,$a_{g,X}b_{g,Y}=\pm 1$. However because $M$ is orientable we have that any combination is possible except if they are all $1$ or all $-1$, because this would contradict that $M$ is orientable with our given $\Sigma$-mapping, and therefore it is clear that
$$a_{e,X}b_{e,Y}+a_{f,X}b_{f,Y}+a_{g,X}b_{g,Y}=\pm 1.$$
Hence we always have that,
$$0=(a_{e,X}b_{e,Y}+a_{f,X}b_{f,Y}+a_{g,X}b_{g,Y}-1)(a_{e,X}b_{e,Y}+a_{f,X}b_{f,Y}+a_{g,X}b_{g,Y}+1)$$
$$=(a_{e,X}b_{e,Y}+a_{f,X}b_{f,Y}+a_{g,X}b_{g,Y})^2-1$$
$$=a_{e,X}^2b_{e,Y}^2+a_{f,X}^2b_{f,Y}^2+a_{g,X}^2b_{g,Y}^2+2a_{e,X}b_{e,Y}a_{f,X}b_{f,Y}+2a_{f,X}b_{f,Y}a_{g,X}b_{g,Y}+2a_{g,X}b_{g,Y}a_{e,X}b_{e,Y}-1$$
$$=1^21^2+1^21^2+1^21^2+2a_{e,X}b_{e,Y}a_{f,X}b_{f,Y}+2a_{f,X}b_{f,Y}a_{g,X}b_{g,Y}+2a_{g,X}b_{g,Y}a_{e,X}b_{e,Y}-1$$
$$=2a_{e,X}b_{e,Y}a_{f,X}b_{f,Y}+2a_{f,X}b_{f,Y}a_{g,X}b_{g,Y}+2a_{g,X}b_{g,Y}a_{e,X}b_{e,Y}+2,$$

and so (because the characteristic of $K$ is not 2), dividing by $2$ we get that
$$0=a_{e,X}b_{e,Y}a_{f,X}b_{f,Y}+a_{f,X}b_{f,Y}a_{g,X}b_{g,Y}+a_{g,X}b_{g,Y}a_{e,X}b_{e,Y}+1=h_{X,Y}.$$
Thus we have our desired solution to our system of equations.
\\

\fussy Conversely, assume we have a zero solution to our system of equations. Then, because for all $e\in X$ and $X\in \calc(M)$ we have $a_{e,X}^2-1=0$, this implies that $a_{e,X}=\pm 1$ and similarly $b_{e,Y}=\pm 1$ for all $e\in Y$ and $Y\in \calc^*(M)$. Then we define a $\Sigma$-mapping of $M$ as follows: for all $X\in \calc(M)$, if $e\in X$ let $\phi_{\Sigma}(e,X)=\text{sgn}(a_{e,X})$ and $\phi_{\Sigma}(e,X)=0$ if $e$$\notin$$X$ and similarly for all $Y\in \calc^*(M)$, if $e\in Y$ let $\phi_{\Sigma}^*(e,Y)=\text{sgn}(b_{e,Y})$ and $\phi_{\Sigma}^*(e,Y)=0$ if $e$$\notin$$Y$, where $\text{sgn}(1)=+$ and $\text{sgn}(-1)=-$ (we note that because $K$ is not a characteristic-2 field, this is well-defined). It is clear by definition that this is a $\Sigma$-mapping.
\fussy Then we must show this mapping satisfies the given condition to show that $M$ is orientable. Suppose there exists an $X\in \calc(M)$, $Y\in \calc^*(M)$ where the given condition is not satisfied.

\fussy First case: If $|X\cap Y|=2$, $X\cap Y=\{e,f\}$. Because we have that $\phi_{\Sigma}(e,X)\phi_{\Sigma}^*(e,Y),$ $\phi_{\Sigma}(f,X)\phi_{\Sigma}^*(f,Y)\neq 0$ (as $e,f\in X\cap Y$), this implies that they are both positive or both negative, which would imply (because our sgn function is multiplicative as defined) that $a_{e,X}b_{e,Y}$ and $a_{f,X}b_{f,Y}$ are both $1$ or both $-1$. Because $K$ is not of characteristic 2 we would have that
$$h'_{X,Y}=a_{e,X}b_{e,Y}+a_{f,X}b_{f,Y}=\pm 2\neq 0,$$
which would be a contradiction.

\fussy Second case: If $|X\cap Y|=3$, $X\cap Y=\{e,f,g\}$. Because we have that $\phi_{\Sigma}(e,X)\phi_{\Sigma}^*(e,Y),$ $\phi_{\Sigma}(f,X)\phi_{\Sigma}^*(f,Y),$
and $\phi_{\Sigma}(g,X)\phi_{\Sigma}^*(g,Y)\neq 0$, then they all must be positive or all must be negative and thus
$a_{e,X}b_{e,Y}$, $a_{f,X}b_{f,Y}$ and $a_{g,X}b_{g,Y}$ are all $1$ or all $-1$, and therefore
$$a_{e,X}b_{e,Y}+a_{f,X}b_{f,Y}+a_{g,X}b_{g,Y}=\pm 3.$$
Thus
$$a_{e,X}b_{e,Y}+a_{f,X}b_{f,Y}+a_{g,X}b_{g,Y}+1=-2\text{ or }4,$$
$$a_{e,X}b_{e,Y}+a_{f,X}b_{f,Y}+a_{g,X}b_{g,Y}-1=-4\text{ or }2,$$
and because $K$ is a field (necessarily with prime characteristic) not of characteristic $2$, then $-4,-2,2,4\neq 0$, and hence
$$0\neq (a_{e,X}b_{e,Y}+a_{f,X}b_{f,Y}+a_{g,X}b_{g,Y}-1)(a_{e,X}b_{e,Y}+a_{f,X}b_{f,Y}+a_{g,X}b_{g,Y}+1)$$
$$=(a_{e,X}b_{e,Y}+a_{f,X}b_{f,Y}+a_{g,X}b_{g,Y})^2-1$$
$$=a_{e,X}^2b_{e,Y}^2+a_{f,X}^2b_{f,Y}^2+a_{g,X}^2b_{g,Y}^2+2a_{e,X}b_{e,Y}a_{f,X}b_{f,Y}+2a_{f,X}b_{f,Y}a_{g,X}b_{g,Y}+2a_{g,X}b_{g,Y}a_{e,X}b_{e,Y}-1$$
$$=1^21^2+1^21^2+1^21^2+2a_{e,X}b_{e,Y}a_{f,X}b_{f,Y}+2a_{f,X}b_{f,Y}a_{g,X}b_{g,Y}+2a_{g,X}b_{g,Y}a_{e,X}b_{e,Y}-1$$
$$=2a_{e,X}b_{e,Y}a_{f,X}b_{f,Y}+2a_{f,X}b_{f,Y}a_{g,X}b_{g,Y}+2a_{g,X}b_{g,Y}a_{e,X}b_{e,Y}+2$$
$$=2(a_{e,X}b_{e,Y}a_{f,X}b_{f,Y}+a_{f,X}b_{f,Y}a_{g,X}b_{g,Y}+a_{g,X}b_{g,Y}a_{e,X}b_{e,Y}+1)=2h_{X,Y}=0,$$
which is again a contradiction. Thus we get that $M$ must be orientable.
\end{proof}

\noindent {\bf Remark:} It is important to notice that the subsystem consisting of all the equations with the exception of the $h_{X,Y}$ polynomials gives
a system that is feasible if and only  if the matroid is weakly-orientable. This is a degree-two system, and, as such, it is typically harder to solve than
the Bland-Jensen linear system.

\begin{cor}\label{eq:polyup} If the polynomial $h'_{X,Y}$ for every $X\in \calc(M)$ and $Y\in \calc^*(M)$ such that $X\cap Y=\{e,f\}$ is replaced with the polynomial equation
\[
\tag{$h_{X,Y}$} a_{e,X}b_{e,Y}a_{f,X}b_{f,Y}+1=0,
\]
in the system described in Theorem \ref{eq:poly}, then $M$ is orientable if and only if this new system has a solution over the field $K$.
\end{cor}
\begin{proof} We have a solution over a field $K$ to our polynomial system in \ref{eq:poly} if and only if  for every $X\cap Y=\{e,f\}$ we have that
$$0=(a_{e,X}b_{e,Y}+a_{f,X}b_{f,Y})^2=a_{e,X}^2b_{e,Y}^2+2a_{e,X}b_{e,Y}a_{f,X}b_{f,Y}+a_{f,X}^2b_{f,Y}^2$$
$$=1^21^2+2a_{e,X}b_{e,Y}a_{f,X}b_{f,Y}+1^21^2=2+2a_{e,X}b_{e,Y}a_{f,X}b_{f,Y}~.$$
Because two is invertible in $K$ (since we are in a field not of characteristic two), this is true if and only if
\begin{align*}
h_{X,Y}&=a_{e,X}b_{e,Y}a_{f,X}b_{f,Y}+1=0.
\end{align*}
\end{proof}

Thus, by Corollary \ref{eq:polyup}, we have that the polynomials $h_{X,Y}$ and $h'_{X,Y}$, in Theorem \ref{eq:poly}, may be replaced with a single concise set of polynomials as follows: if $X\in \calc(M)$ and $Y\in \calc^*(M)$ with $|X\cap Y|\in\{2,3\}$, then
$$h_{X,Y}=1+\sum_{\{e,f\}\in X\cap Y}a_{e,X}b_{e,Y}a_{f,X}b_{f,Y}.$$

We observe that our polynomial systems behave well with respect to duality and minors:

\begin{lem}\label{eq:dual} Consider a matroid $M$
\begin{itemize}
\item If $M^*$ is its dual matroid, then for any field and the corresponding polynomial systems, described either in Theorems \ref{eq:poly} or \ref{eq:poly2}, for $M$ and $M^*$ are the same, up to relabeling of variables.
\item If $M'$ be a minor of $M$,  then for any field and the corresponding polynomial system, described either in Theorems \ref{eq:poly} or \ref{eq:poly2}, for $M'$ is a subsystem of the one associated to $M$, up to relabeling.
\end{itemize}
\end{lem}

\begin{proof} The first statement is obvious. For the second, it clearly suffices to
check that the result hold when we contract a single element. Hence assume $M'=M/f$. Recall that $\calc(M/f)=\min\{C\setminus f\big| C\in\calc(M)\}$ and $\calc^*(M/f)=\calc(M^*\setminus f)=\{C\big| C\in\calc^*(M), f\notin C\}$ \cite{Ox}[~Sec. 3.1]. Let $X'\in\calc(M/f)$ and $Y'\in\calc^*(M/f)$, then $Y'\in\calc^*(M)$ and we may associate a $X\in\calc(M)$ to each $X'$ such that $X'=X\setminus f$.

We associate the variables $a'_{e,X'}$ and $b'_{e,Y'}$ in the system for $M'=M/f$ to the variables $a_{e,X}$and $b_{e,Y'}$ in the system for $M$. It is immediately clear that in the case of Theorem \ref{eq:poly} the $p'_{e,X'}$ and $q'_{e,Y}$ polynomials in the system for $M'$ are contained in the system for $M$. Further we have that $X'\cap Y'=X\cap Y'$ because $X'=X\setminus e$ and $e\notin Y'$, which immediately implies also that the remaining polynomials in the system for $M'$ are contained in the system for $M$.
\end{proof}

\subsection{Remarks on computational complexity}

The fact that there are many characterizations of matroids is a very appealing property. Each of these different characterizations
carries with it a natural way to encode the same matroid as input for algorithmic purposes, either explicitly or by an oracle:
 via its rank function, independent sets, bases, circuits, cocircuits, flats, by a matrix representation (when one exists), etc. In some
 cases, the difference in the sizes of  two encodings for a matroid can be dramatic. For example, if a matroid $M$ with ground set
 $E$ is representable over the rationals, the input can  just be an integer $r \times n$ matrix, where $r=rank(M)$ and $n=|E|$, and
 there is a fast procedure to decide when a subset is independent in $M$. Compare this encoding to an encoding of $M$ via a list
 of its circuits. The number of circuits of $M$ can be of order $O(n^{r+1})$.
This disparity between the different input sizes can give the wrong impression that with large encodings, such as the entire list
of bases, one is certain to have trivial polynomial-time complexity. But this is far from the truth. For example, Richter-Gebert proved
that deciding orientability for matroids is NP-complete even for co-rank-three matroids encoded by their bases (see \cite{OrientNP}).

Because of Richter-Gebert's theorem, deciding whether the polynomial system described in Theorem $\ref{eq:poly2}$ over $F_2$
or the polynomial system described in Theorem $\ref{eq:poly}$ over any field of characteristic different two is feasible are NP-complete
problems. Of course, the number of variables of these systems is the sum of the cardinalities of all circuits and cocircuits.

While the system is large, one of the promising aspects of weak orientability is its equivalence to the solvability of the Bland-Jensen linear system. Note that every solution to the Bland-Jensen linear system corresponds to a weak orientation of the matroid. Conversely, every weak orientation of the matroid has a solution to the Bland-Jensen linear system corresponding to it. This correspondence is not one-to-one, as multiple $\Sigma$-mappings, and thus solutions, correspond to the same weak orientation. However we have that the set of all solutions to the Bland-Jensen linear system, under an equivalence relation, is exactly the set of all weak orientations of the matroid. Since we can get a parametric description of all solutions to a linear system, by finding one solution and a basis for the null space of the equation matrix, we can efficiently calculate a simple description of all the weak orientations of the matroid. Finally, because solving systems of linear equations can be carried out efficiently,  we get the following:

\begin{cor}\label{weakt} Given a matroid $M$, specified by the sets of all its circuits and cocircuits,  we can determine in polynomial time, whether  or not
$M$ is weakly-orientable. In the case that it is weakly-orientable, we can give a compact description of all of its weak orientations.
\end{cor}

Again, we see this as a useful result for arbitrary matroids, because in some circumstances a small encoding of a matroid is simply not at hand   (e.g., some matroids are not representable over any field) and may not even be possible.  But there are classes of matroids that have short descriptions and could be used for computing much faster. For example, in 1975 Seymour  \cite{seymour} proved that  a matroid is binary if and only if $|X\cap Y|\neq 3$ for all pairs of circuits and cocircuits $X\in\calc(M)$ and $Y\in\calc^*(M)$.
Thus a binary matroid is weakly-orientable if and only if it is orientable. While deciding whether a binary matroid is orientable can be determined using Corollary \ref{weakt}, this is
perhaps not as efficient as applying deep results of Seymour: as a consequence of Seymour's work on regular matroids (see \cite{Decomp}), we can already efficiently determine orientability directly from the $F_2$ matrix representing the matroid. In conclusion, while in some cases our system of equations could be  considered too large, in others, it may be quite practical. See \cite{HausKort} and \cite{Mayhew08} for discussions of matroid axioms in relation to computational complexity.

\section{A key example: The Fano Matroid} \label{sec_fano}

The Fano Matroid, derived from the Fano Plane (the finite projective plane of order 2, having the smallest possible number of points and lines) is a non-orientable matroid on seven elements \cite{MinZieg}. Since all matroids on six or less elements are orientable, this is a classical example of a minimal non-orientable matroid. Here we use our systems of polynomial equations to verify that the Fano matroid is indeed non-orientable.
\\

The Fano matroid is the matroid on the elements $\{1,2,3,4,5,6,7\}$, where every three element subset is a basis with the exception of the sets $\{1,2,3\}$, $\{1,6,7\}$, $\{3,5,6\}$, $\{2,5,7\}$, $\{3,4,7\}$, $\{1,4,5\}$, and $\{2,4,6\}$. Here are the circuits and cocircuits of this matroid:
\begin{center}
\begin{tabular}{cc|c}
\multicolumn{2}{c|}{Circuits} & Cocircuits\\ \hline
$\{1,6,7\}\spc$ & $\{4,5,6,7\}\spc$ & $\{2,3,4,5\}$\\
$\{3,5,6\}\spc$ & $\{2,3,6,7\}\spc$ & $\{1,2,4,7\}$\\
$\{2,5,7\}\spc$ & $\{1,2,5,6\}\spc$ & $\{1,3,5,7\}$\\
$\{3,4,7\}\spc$ & $\{1,3,4,6\}\spc$ & $\{1,3,4,6\}$\\
$\{1,4,5\}\spc$ & $\{1,3,5,7\}\spc$ & $\{1,2,5,6\}$\\
$\{2,4,6\}\spc$ & $\{1,2,4,7\}\spc$ & $\{2,3,6,7\}$\\
$\{1,2,3\}\spc$ & $\{2,3,4,5\}\spc$ & $\{4,5,6,7\}$\\

\end{tabular}
\end{center}
From this list of circuits and cocircuits, we see the following table of intersection cardinalities:
\begin{center}
\includegraphics[scale=0.65]{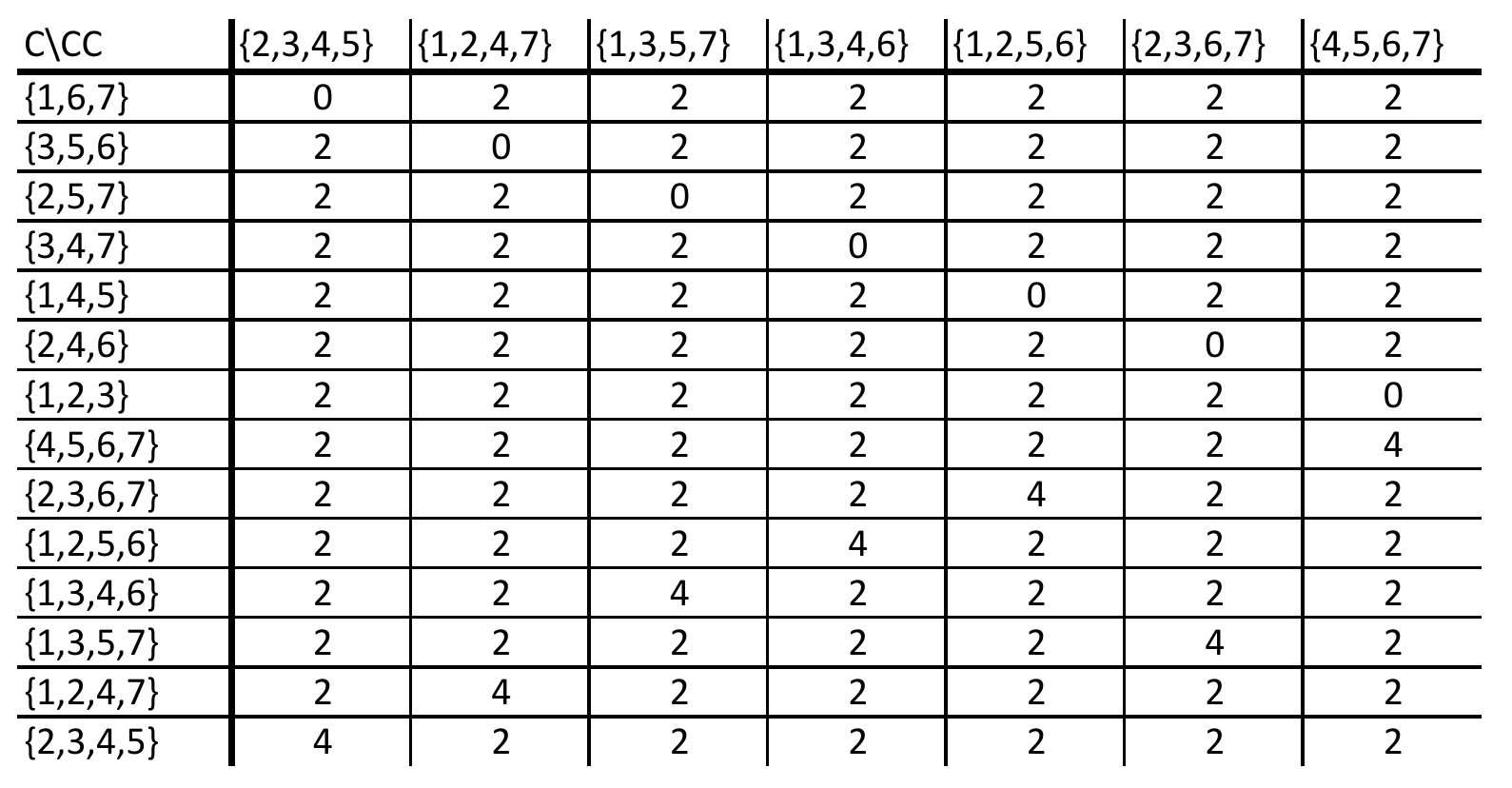}

\end{center}

We observe from this table of intersections that the matroid is binary (because every circuit/cocircuit intersection has even cardinality), and thus we 
can determine its orientability via the linear system over $F_2$ defined in Thm. \ref{eq:polyweak}.
Since this system can easily be checked for a solution using any linear system solver over
 $F_2$, we have an easy and compact proof that the Fano Matroid is non-orientable.

Consider a proof of this via the famous Hilbert's Nullstellensatz  (see \cite{Eisen}). Given an algebraically-closed field $\mathbb{K}$ and a set of polynomials $f_1,\ldots,f_s \in \mathbb{K}[x_1,\ldots,x_n]$, Hilbert's Nullstellensatz states that the system of polynomial equations $f_1 = f_2 = \cdots = f_s = 0$ has \emph{no} solution if and only if there exist polynomials $\beta_1,\ldots,\beta_s \in \mathbb{K}[x_1,\ldots,x_n]$ such that

$1 = \sum_{i=1}^{s}\beta_if_i~.$ We measure the complexity of a given certificate in terms of the size of the $\beta$ coefficients, because these are the unknowns we must discover in order to demonstrate the \emph{non}-existence of a solution to $f_1 = f_2 = \cdots = f_s = 0$. Thus, we measure the degree of a Nullstellensatz certificate as $d = \max\{\deg(\beta_1),\ldots, \deg(\beta_s)\}$.

Consider the following \emph{degree-zero} Hilbert's Nullstellensatz certificate for the \emph{non}-orientability of the Fano matroid:
\begin{align*}
1 &= (a_{6,167}+a_{7,167}+b_{6,2367}+b_{7,2367}+1) +(a_{6,167}+a_{7,167}+b_{6,4567}+b_{7,4567}+1)\\
&\phantom{=}\hspace{3pt} + (a_{3,356}+a_{5,356}+b_{3,1357}+b_{5,1357}+1) + (a_{3,356}+a_{6,356}+b_{3,2367}+b_{6,2367}+1)\\
&\phantom{=}\hspace{3pt} + (a_{5,356}+a_{6,356}+b_{5,4567}+b_{6,4567}+1) + (a_{5,257}+a_{7,257}+b_{5,1357}+b_{7,1357}+1)\\
&\phantom{=}\hspace{3pt}  + (a_{5,257}+a_{7,257}+b_{5,4567}+b_{7,4567}+1) + (a_{3,347}+a_{7,347}+b_{3,1357}+b_{7,1357}+1)\\
&\phantom{=}\hspace{3pt}  + (a_{3,347}+a_{7,347}+b_{3,2367}+b_{7,2367}+1)  \hspace{5pt}\text{mod}~2~.
\end{align*}

This certificate (and the others that follow) are found via the \emph{Nullstellensatz Linear Algebra algorithm} (or NulLA, see  \cite{Nulla2} for details). Roughly speaking NulLA works as follows: Given a system of polynomial equations, we fix a tentative degree $d$ for the certificate meaning
$\deg(\beta_if_i) = d$ for every $i=1,...,s$. We can decide whether there is a Nullstellensatz certificate of degree $d$
by solving a system of \emph{linear} equations over the field $\mathbb{K}$ whose variables are in bijection with the coefficients of the monomials of
the polynomials $\beta_1,\dots,\beta_s$. If this linear system has a solution, we have found a certificate; otherwise, we try a higher
degree for the certificate. This process is guaranteed to terminate because, for a Nullstellensatz certificate to exist, the degrees of the certificate cannot be more than known bounds.

Fano is non-orientable, thus our polynomial systems detecting orientability must be infeasible, thus a Nullstellensatz certificate must exist.
We experimented a bit with how much the degree of the coefficients grow. We considered the system of polynomial equations from Theorem \ref{eq:poly} over the field $F_3$. Our first attempt to prove the non-orientability of the Fano matroid via NulLA yielded a certificate of degree four or greater. Despite allocating 12GB of RAM on a high-performance computing cluster, we were unable to explicitly determine a certificate beyond this bound.

We now proceed to described a modified version of the system over $F_3$ (or any field), for which NulLA is able to find a certificate of smaller degree. In order to simplify the original system, we observe that certain variables in our system may be fixed to a constant value of one.
We note that for a given $\Sigma$-mapping on a matroid $M$ satisfying Definition \ref{eq:ordef}, we may flip the sign of this $\Sigma$-mapping on all the elements in a single circuit or cocircuit. Transferring this change to our polynomial system in the case of fields  of characteristic different from two, for any solution to our system and any circuit $X\in\calc(M)$ or cocircuit $Y\in\calc^*(M)$, we may replace it with a new solution with $a_{e,X}$ changed to $-a_{e,X}$ for all $e\in X$ or with $b_{e,Y}$ changed to $-b_{e,Y}$ for all $e\in Y$. Also note that for a given (weak) orientation, we may flip the sign of all the circuit or cocircuit values for a particular element $e\in E$ and obtain a new (weak) orientation (see \cite{Matorient},\cite{Weak}). Correspondingly, for any solution of our polynomial system, for a fixed $e\in E$, we may replace $a_{e,X}$ with $-a_{e,X}$ or $b_{e,X}$ with $-b_{e,X}$ and obtain a new solution.

Using these exchanges, we may assume that values of certain variables in our system are fixed to $1$ if a solution exists, and thus we can do the following variable fixings:
\begin{itemize}
\item For each circuit $X\in\calc(M)$, choose an $e\in X$, and fix $a_{e,X}=1$.
\item For each cocircuit $Y\in\calc^*(M)$, choose an $e\in Y$, and fix $b_{e,Y}=1$.
\item For each element $e\in E$ such that $e$ was not an element chosen in any of the previous variable fixings, choose $X\in\calc(M)$ such that $e\in X$, and fix $a_{e,X}=1$.
\end{itemize}
Such variable fixings are common and have been utilized in previous computational problems for resolving orientability questions on matroids (see \cite{10ele}). These extra restrictions allow us to find an explicit degree-three Nullstellensatz certificate for the non-orientability of the Fano matroid in just over one second of computing time.

\begin{align*}
1 &= (-b_{6,1346}b_{6,1256})(a_{6,167}^2+2) + (b_{6,1346}b_{6,1256})(a_{6,356}^2+2) + (b_{3,1346}b_{5,1256})(a_{5,145}^2+2)\\
&\phantom{=}\hspace{3pt}  + (1+a_{5,257}a_{7,347})(b_{3,1357}^2+2) + (-a_{5,257}a_{7,347}+-a_{5,145}b_{3,1346})(b_{5,4567}^2+2)\\
&\phantom{=}\hspace{3pt}  + (a_{6,167}b_{6,1256})(a_{6,167}b_{6,1346}+1) + 2(a_{6,167}b_{6,1256}+1)\\
&\phantom{=}\hspace{3pt}  + (-a_{5,257}a_{7,347}b_{3,1357})(b_{3,1357}+b_{5,1357}) + (-a_{6,356}b_{6,1256})(a_{6,356}b_{6,1346}+b_{3,1346})\\
&\phantom{=}\hspace{3pt}  + (b_{3,1346})(a_{6,356}b_{6,1256}+b_{5,1256}) + (a_{7,347}b_{3,1357})(a_{5,257}b_{5,1357}+b_{7,1357})\\
&\phantom{=}\hspace{3pt}  + (a_{7,347}b_{5,4567})(a_{5,257}b_{5,4567}+b_{7,4567}) + (-b_{3,1357})(a_{7,347}b_{7,1357}+b_{3,1357})\\
&\phantom{=}\hspace{3pt}  + (-a_{4,145}b_{5,4567})(a_{4,347}b_{4,1346}+b_{3,1346}) + (-b_{5,4567})(a_{4,347}+a_{7,347}b_{7,4567})\\
&\phantom{=}\hspace{3pt} + (a_{4,347}b_{5,4567})(a_{4,145}b_{4,1346}+1) + (-a_{5,145}b_{3,1346})(a_{5,145}b_{5,1256}+1)\\
&\phantom{=}\hspace{3pt}  + (b_{3,1346}b_{5,4567})(a_{4,145}+a_{5,145}b_{5,4567}) \hspace{5pt}\text{mod}~3~.
\end{align*}

As a final application to the Fano Matroid, it is interesting to find minimal infeasible subsystems of the linear system associated to this matroid. This allows us to find patterns that are sufficient for non-weak orientability.

\begin{lem}\label{eq: Syst1} Let $M$ be a matroid, and $F\subseteq E(M)$ with $|F|=4$, with $A_i$, $1\leq i\leq 4$, being the three element subsets of $F$ and $B_j$, $1\leq j\leq 3$, being the two element subsets of $F$ not contained in $A_4$. If there exist circuits $C_{A_4}$, $C_{B_j}$, for all $j$, and cocircuits $C_{A_i}^*$, for $1\leq i\leq 3$ such that
\begin{enumerate}[i)]
\item $C_{A_i}^*\cap C_{B_j}=B_j,\spc\text{ for all }B_j\subseteq A_i$,
\item $C_{A_i}^*\cap C_{A_4}=A_i\cap A_4,\spc\text{ for all }1\leq i\leq 3$,
\end{enumerate}
then $M$ is non-weakly-orientable.
\end{lem}

\begin{proof} Consider all circuit/cocircuit intersections noted in the statement of the lemma as a list of circuit/cocircuit pairs. It is straightforward to check that this list is of the type described in Theorem \ref{eq:weakorlem}, and thus they are sufficient to conclude the matroid to be non-weakly-orientable.
\end{proof}

The usefulness of this sufficient criterion will become apparent in the next section.

\section{An Infinite Family of Non-Weakly Orientable Matroids}

Let $n\geq 0$, and define
$$E_n=\{1,2,3,4,x_0,x_1,...,x_n,y_0,y_1,...,y_n,z_0,z_1,...,z_n\}$$
to be a set of cardinality $3n+7$, and further for a given $n$ let $X=\{x_0,x_1,...,x_n\}$, $Y=\{y_0,y_1,...,y_n\}$ and $Z=\{z_0,z_1,...,z_n\}$. Then define the sets
\\\\\\
\fussy$$H_1=\{1\}\cup Y\cup Z,\spc\spc H_2=\{2\}\cup X\cup Z,\spc\spc H_3=\{3\}\cup X \cup Y,$$
$$C_1=\{1,4\}\cup X,\spc\spc C_2=\{2,4\}\cup Y,\spc\spc C_3=\{3,4\}\cup Z,$$
and the collection of subsets
$$\mathcal{B}_n=\{B\subseteq E_n\big|\spc |B|=n+3,\spc\{1,2,3\}\nsubseteq B,\spc B\nsubseteq C_i,\text{ and }B\nsubseteq H_i\text{ for all }1\leq i\leq 3\}.$$
\begin{lem}\label{eq:matfam} For $n\geq 0$ the pair $(E_n,\mathcal{B}_n)$, henceforth denoted $M\!I_n$, defines a matroid of rank $n+3$ on $3n+7$ elements, with $\mathcal{B}_n$ being the bases of this matroid.
\end{lem}

\begin{proof} \fussy We must show that $\mathcal{B}_n$ satisfies the base axiom for a matroid. Fix $A,B\in\mathcal{B}_n$ and $\alpha\in A\setminus B$, then for any $\beta\in B\setminus A$ define $S_\beta :=(A\setminus\{\alpha\})\cup\{\beta\}$. To complete the proof, we must show $S_\beta\in\mathcal{B}_n$ for some $\beta\in B\setminus A$.
\\\\
\fussy Clearly, because $|A|=|B|=n+3$ and $A\setminus B\neq \varnothing$, we have that $B\setminus A$ is non-empty, so there exists some $\beta\in B\setminus A$ so that $S_\beta$ is well-defined and $|S_\beta|=|A|=|B|=n+3$.
\\\\
\fussy Now we show  for some $\beta\in B\setminus A$ that $S_\beta \nsubseteq C_i$ and $S_\beta \nsubseteq H_i$ for $1\leq i\leq 3$. Suppose not; then for all $\beta\in B\setminus A$, we have that $S_\beta \subseteq C_i$ or $S_\beta \subseteq H_i$ for some $1\leq i\leq 3$.

\fussy Now suppose for a given $\beta\in B\setminus A$ that $A\setminus\{\alpha\}\subseteq S_\beta \subseteq C_i$ for some $i$. Because $|A\setminus\{\alpha\}|=n+2$ and $|C_i\cap C_k|=1$ for $i\neq k$, we have that $A\setminus\{\alpha\}\nsubseteq C_k$ for $i\neq k$. Further, because $|C_i\cap H_k|=n+1$ for $i\neq k$ and $|C_i\cap H_i|=1$, we have that $A\setminus\{\alpha\}\nsubseteq H_j$ for all $j$.

\fussy Similarly, suppose that for a given $\beta\in B\setminus A$ that $A\setminus\{\alpha\}\subseteq S_\beta \subseteq H_i$ for some $i$. Again, because $|A\setminus\{\alpha\}|=n+2$ and $|H_i\cap H_k|=n+1$ for $i\neq k$, we have that $A\setminus\{\alpha\}\nsubseteq H_k$ for $i\neq k$. Further, because $|H_i\cap C_k|=n+1$ for $i\neq k$ and $|H_i\cap C_i|=1$, we have that $A\setminus\{\alpha\}\nsubseteq C_j$ for all $j$.

\fussy In either case, we get that there is exactly one of $H_1,H_2,H_2,C_1,C_2,$ and $C_3$ which contains $A\setminus\{\alpha\}$; we denote this distinguished set by $K$. However, because by supposition for all $\beta\in B\setminus A$, we have $S_\beta$ is contained in at least one of these six sets and that $A\setminus\{\alpha\}\subseteq S_\beta$ we get that $S_\beta\subseteq K$ for all $\beta\in B\setminus A$. This implies for all $\beta\in B\setminus A$ that $\beta\in S_\beta\subseteq K$ so that $B\setminus A\subseteq K$. Finally, because $\alpha\notin B$, we have that $B\subseteq(B\setminus A)\cup (A\setminus\{\alpha\})\subseteq K$. This is a contradiction because $B\in\mathcal{B}_n$. Hence, there exists some $\beta\in B\setminus A$ so that $S_\beta \nsubseteq C_i$ and $S_\beta \nsubseteq H_i$ for $1\leq i\leq 3$.
\\\\
\fussy Therefore, let $\overline{B\setminus A}$ denote the $\beta\in B\setminus A$ such that $S_\beta\nsubseteq C_i$ and $S_\beta\nsubseteq H_i$ for $1\leq i\leq 3$. We have that this is non-empty by the above, and therefore to show that there exists a $\beta\in B\setminus A$ such that $S_\beta\in \mathcal{B}_n$, it suffices to show there exists a $\beta\in \overline{B\setminus A}$ such that $\{1,2,3\}\nsubseteq S_\beta$. Suppose not; then because we have for any $\beta\in \overline{B\setminus A}$ that $\{1,2,3\}\subseteq S_\beta$ and also that $A\in\mathcal{B}_n$ so that $\{1,2,3\}\nsubseteq A\setminus\{\alpha\}\subseteq A$, we get that $|\{1,2,3\}\cap (A\setminus\{\alpha\})|=2$ and for all $\beta\in \overline{B\setminus A}$ that $\beta\in\{1,2,3\}$. This implies that $\overline{B\setminus A}\subseteq\{1,2,3\}$ and because $(A\setminus\{\alpha\})\cap \overline{B\setminus A}=\varnothing$ we get that $\overline{B\setminus A}=\{e\}$ where $e\in\{1,2,3\}\setminus(A\setminus\{\alpha\})$. However because $|\{1,2,3\}\cap (A\setminus\{\alpha\})|=2$ whereas $|\{1,2,3\}\cap C_i|=|\{1,2,3\}\cap H_i|=1$ for all $1\leq i\leq 3$ we also get that $A\setminus\{\alpha\}\nsubseteq C_i$ and $A\setminus\{\alpha\}\nsubseteq H_i$ for all $i$ and therefore $S_\beta \nsubseteq C_i$ and $S_\beta\nsubseteq H_i$ for all $i$. This implies that $B\setminus A=\overline{B\setminus A}=\{e\}$. Then we have that $B\subseteq(B\setminus A)\cup (A\setminus\{\alpha\})=(A\setminus\{\alpha\})\cup\{e\}$. However because $e\notin A$, $\alpha\notin B$, and $|A|=|B|$ we have that $B=(A\setminus\{\alpha\})\cup\{e\}$, but $\{1,2,3\}\subseteq (A\setminus\{\alpha\})\cup\{e\}$ which is a contradiction because $B\in\mathcal{B}_n$. Hence there exists $\beta\in \overline{B\setminus A}$ such that $\{1,2,3\}\nsubseteq S_\beta$ and we get that $S_\beta\in\mathcal{B}_n$, thus completing the proof that $(E_n,\mathcal{B}_n)$ is a matroid. That the rank of this matroid is $n+3$ follows immediately.

\end{proof}

Note that the matroid $M\!I_0$ is simply the Fano Matroid, so this family extends the Fano Matroid to a large family of matroids.

\begin{lem} For $n\geq0$ the matroids $M\!I_n$ are non-weakly-orientable.
\end{lem}

\begin{proof} We will show that the $M\!I_n$ satisfy the conditions of Lemma \ref{eq: Syst1}.
\\\\
\fussy To do this first we note for any set $I\subseteq E_n$ with $|I|=n+2$ and $\{1,2,3\}\nsubseteq I$ that $I$ is a independent set. For if $|I\cap\{1,2,3\}|=2$ then for any $e\in E_n\setminus\{1,2,3\}$ we have $I\cup\{e\}$ is a basis and if $|I\cap\{1,2,3\}|=1$ then for any $f\in\{1,2,3\}$ with $f\notin I$ we have $I\cup\{f\}$ is a basis.

 If $I\cap\{1,2,3\}=\varnothing$ then we have two cases. If $4\in I$ then $I$ must also intersect at least one of $X,Y,\text{ or }Z$, so by choosing $e\in X\cup Y\cup Z$ so that $e$ is not in one of the sets $I$ intersects then $I\cup\{e\}$ will not be contained in any of the $C_i$ or $H_i$ and will be a basis. If $4\notin I$ then $I$ must intersect at least two of $X,Y,\text{ or }Z$, so by choosing $e\in X\cup Y\cup Z$ so that $I\cup\{e\}$ intersects all three sets $X,Y,\text{ and }Z$ we get that it is a basis as well.
\\\\
\fussy Now for any $\{e,f\}\subseteq\{1,2,3\}$  we have $\{e,f\}\cup (X\setminus\{x_n\})$ satisfies the above conditions and hence is independent so that $\{e,f\}$ is independent and because $\{1,2,3\}$ is clearly dependent because it cannot be contained in any basis we have that $\{1,2,3\}$ is a circuit.

\fussy Since $|C_i|=n+3$ and $\{1,2,3\}\nsubseteq C_i$ for all $1\leq i\leq 3$, we have any proper subset of any $C_i$ is independent and because $C_i$ is clearly dependent for all $i$, we have that $C_i$ is a circuit of $M\!I_n$ for all $i$.

\fussy Further because $|H_i|=2n+3$ and $\{1,2,3\}\nsubseteq H_i$ for all $1\leq i\leq 3$ we have that $H_i$ contains a cardinality $n+2$ independent set for all $i$, but because any $H_i$ clearly does not contain a basis we have that $\text{rank}(H_i)=n+2$ for all $i$. These $H_i$ are also closed sets and thus hyperplanes, to see this denote $K_1=X, K_2=Y,$ and $K_3=Z$. Then for any $e\notin H_i$ if $e=4$ then $\{i,e\}\cup K_j\subseteq H_i\cup\{e\}$ is a basis for $j\neq i$, if $e\in\{1,2,3\}\setminus\{i\}$ then $\{i,e\}\cup K_j\subseteq H_i\cup\{e\}$ is a basis for $j\neq i$, and if $e\in K_i$ then $\{e,f\}\cup K_j\subseteq H_i\cup\{e\}$ is a basis for $j\neq i$ and $f\in K_k$ where $k\neq i,j$. So for all $e\notin H_i$ we have $\text{rank}(H_i\cup\{e\})=n+3$ and thus $H_i$ is a hyperplane for all $i$. Then because cocircuits are exactly complements of hyperplanes we have that $H_i^c$ is a cocircuit of $M\!I_n$ for all $i$.
\\\\
\fussy Then letting $F=\{1,2,3,4\}$, $A_1=\{2,3,4\}$,$A_2=\{1,3,4\}$,$A_3=\{1,2,4\}$,$A_4=\{1,2,3\}$, $B_1=\{1,4\}$,$B_2=\{2,4\}$, and $B_3=\{3,4\}$ as in Lemma \ref{eq: Syst1}, we have circuits $C_{A_4}=\{1,2,3\}$ and $C_{B_i}=C_i$ for all $i$ and cocircuits $C^*_{A_i}=H_i^c$ for all $i$. These can be easily checked to satisfy the conditions of Lemma \ref{eq: Syst1}, and hence by the lemma we get that the $M\!I_n$ is non-weakly-orientable.
\end{proof}

It was originally conjectured by Bland and Jensen (see \cite{Weak}) that weak-orientability cannot be described by a finite list of excluded minors. If the matroids $M\!I_n$ are minor-minimal with respect to weak-orientability, then this would resolve the conjecture in the positive. We have by symmetry in the $M\!I_n$ that
\[
M\!I_n/\{x_n\} \cong M\!I_n/\{e\}\text{ and }M\!I_n\setminus\{x_n\} \cong M\!I_n\setminus\{e\}\text{, for all }e\in X\cup Y\cup Z
\]
\[
M\!I_n/\{1\} \cong M\!I_n/\{e\}\text{ and }M\!I_n\setminus\{1\} \cong M\!I_n\setminus\{e\}\text{, for all }e\in\{1,2,3\},
\]
and thus to check the $M\!I_n$ are minor-minimal it suffices to check that for all $n$: $M\!I_n/\{x_n\},M\!I_n\setminus\{x_n\},M\!I_n/\{1\},M\!I_n\setminus\{1\},M\!I_n/\{4\},$ and $M\!I_n\setminus\{4\}$ are all weakly-orientable. Using this fact we have checked explicitly using the Bland-Jensen linear systems that $M\!I_n$ are minor-minimal for $n\leq 2$. Thus we conjecture:
\begin{conj} For all integers $n\geq 0$, $M\!I_n$ are minor-minimal with respect to being non-weakly-orientable. Hence, weak-orientability cannot be described by a finite list of excluded minors.
\end{conj}

\section{A Classification of Small Weakly-Orientable Matroids} \label{tablesofwoms}

The theoretical discussion in Section \ref{introtowoms} (and a description of the NulLA algorithm in Section \ref{sec_fano}) presents a clear path for efficiently classifying matroids as weakly or non-weakly-orientable. By computing a solution to the Bland-Jensen linear equations, or similarly, presenting a degree one Nullstellensatz certificate or verifying that no such certificate exists, we can quickly test a large number of matroids for weak orientability.

Using the list of all matroids with 12 or fewer elements conveniently generated in \cite{10ele} (as well as the list of \cite{9ele}), we tested all the matroids with 9 elements or less, and all the matroids between 10 and 12 elements having rank three, for weak orientability. The linear systems associated with matroids on 9 elements or less were solvable using MATLAB, but the systems associated with the larger element matroids (10, 11 and 12) were too large and too numerous for MATLAB to effectively handle. However, utilizing the high-performance computing cluster at Penn State University (and the NulLA software optimized for solving systems of linear equations over $F_2$), we were able to process both the 10,037 ten element matroids, and the 298,491 eleven element matroids, in just a few hours. Not surprisingly, processing the 31,899,134 twelve element matroids was more difficult. Using 90 parallel jobs, with each allocated 2GB of RAM, we were able to process all the 12 element matroids in under 12 hours. The Penn State University computing cluster has an Intel Xeon X5675 Six-Core 3.06 GHz processor.

The results are summarized in the following tables. It is evident that about half the time when a matroid
was not orientable, the matroid was already non-weakly-orientable and thus non-realizable over infinitely many finite fields (including $F_3$).
The reader can also access the specific non-weakly-orientable matroids directly from the web site: \\ \url{https://www.math.ucdavis.edu/~jmiller/weakor/}. Any code related to these computations will be made available on request.

\begin{center}
\begin{tabular}{|l||*{6}{c|}}\hline
\multicolumn{7}{|c|}{\makebox{\raisebox{-0.4\height}{\bf Numbers of Non-Weakly Orientable/Total Number of Matroids}}}\\\hline
\backslashbox{Rank}{\raisebox{-0.2\height}{Size}}
&\makebox[3em]{7}&\makebox[3em]{8}&\makebox[3em]{9} &\makebox[3em]{10}&\makebox[3em]{11} &\makebox[3em]{12}\\\hline\hline
&&&&&&\\
\tab 3 & \makebox{$\displaystyle\frac{1}{108}$}& \makebox{$\displaystyle\frac{4}{325}$}&\makebox{$\displaystyle\frac{20}{1275}$}&\makebox{$\displaystyle\frac{172}{10037}$}&\makebox{$\displaystyle\frac{5670}{298491}$}&\makebox{$\displaystyle\frac{1080959}{31899134}$}\\[3ex]\hline
&&&&&&\\
\tab 4 &\makebox{$\displaystyle\frac{1}{108}$}&\makebox{$\displaystyle\frac{31}{940}$}&\makebox{$\displaystyle\frac{7912}{190214}$}&{$\displaystyle\frac{?}{4886380924}$}&$?$&$?$ \\[3ex]\hline
\end{tabular}
\end{center}

\begin{center}
\begin{tabular}{|l||*{6}{c|}}\hline
\multicolumn{7}{|c|}{\makebox{\raisebox{-0.4\height}{\bf Numbers of Simple Non-Weakly Orientable/ Simple Non-Orientable Matroids}}}\\\hline
\backslashbox{Rank}{\raisebox{-0.2\height}{Size}}
&\makebox[3em]{7}&\makebox[3em]{8}&\makebox[3em]{9} &\makebox[3em]{10}&\makebox[3em]{11} &\makebox[3em]{12}\\\hline\hline
&&&&&&\\
\tab 3 & \makebox{$\displaystyle\frac{1}{1}$}& \makebox{$\displaystyle\frac{2}{3}$}&\makebox{$\displaystyle\frac{9}{18}$}&\makebox{$\displaystyle\frac{105}{201}$}&\makebox{$\displaystyle\frac{4877}{9413}$}&\makebox{$\displaystyle\frac{1023104}{1999921}$}\\[3ex]\hline
&&&&&&\\
\tab 4 &\makebox{$\displaystyle\frac{1}{1}$}&\makebox{$\displaystyle\frac{29}{34}$}&\makebox{$\displaystyle\frac{7794}{12284}$}&?&$?$&$?$ \\[3ex]\hline
\end{tabular}
\end{center}

\section{Acknowledgments} We are grateful to Bernd Sturmfels and J\"urgen Richter-Gebert for their useful suggestions and comments.
J.A. De Loera was partially supported by NSF grant DMS--0914107. J. Lee was partially supported by the NSF Grant CMMI--1160915.
J. Miller was partially supported by VIGRE-NSF grant DMS--0636297.

\bibliographystyle{plain}
\bibliography{paperbib}

\end{document}